\documentclass[11pt]{amsart}
\usepackage{amscd,amssymb,verbatim}
\theoremstyle{plain}
\newtheorem{Thm}{Theorem}[section]
\newtheorem{Cor}[Thm]{Corollary}
\newtheorem{Lem}[Thm]{Lemma}
\newtheorem{Prop}[Thm]{Proposition}
\newtheorem{Def}[Thm]{Definition}
\newtheorem{remark}[Thm]{Remark}

\numberwithin{equation}{section}

\begin{document}
\title[New examples of \(c_0\)-saturated Banach Spaces II]
{New examples of \(c_0\)-saturated  Banach spaces II}
\author{I. Gasparis}
\address{Department of Mathematics \\
Aristotle University of Thessaloniki \\
Thessaloniki 54124, Greece.}
\email{ioagaspa@math.auth.gr}
\keywords{\(c_0\)-saturated space, upper \(\ell_p\) estimates, quotient map}
\subjclass{(2000) Primary: 46B03. Secondary: 06A07, 03E02.}
\begin{abstract}
For every Banach space \(Z\) with a shrinking unconditional basis satisfying an upper \(p\)-estimate
for some \(p > 1\), an isomorphically polyhedral Banach space is constructed which has an unconditional
basis and admits a quotient isomorphic to \(Z\). It follows that reflexive Banach spaces with an unconditional
basis and non-trivial type, Tsirelson's original space and \((\sum c_0)_{\ell_p}\) for \(p \in (1, \infty)\),
are isomorphic to quotients of isomorphically polyhedral Banach spaces with unconditional bases.
\end{abstract}
\maketitle
\section{Introduction}
An infinite-dimensional Banach space is \(c_0\)-saturated if every closed, linear, infinite-dimensional
subspace contains a closed, linear subspace isomorphic to \(c_0\). It is classical result \cite{ps}
that every \(C(K)\) space with \(K\) being a countable infinite, compact metric space, is \(c_0\)-saturated.
This result was generalised in \cite{f} to the class of the so-called Lindenstrauss-Phelps
spaces, i.e., spaces whose dual closed unit ball has but countably many 
extreme points. These spaces in turn, belong to the class of the {\em isomorphically polyhedral}  
spaces. We recall that a Banach space is polyhedral if the closed unit ball of each of its finite-dimensional
subspaces has finitely many extreme points. It is isomorphically polyhedral if it is polyhedral under an
equivalent norm. It was proved in \cite{f2} that separable isomorphically polyhedral spaces are \(c_0\)-saturated.

Not much is known about the behavior of isomorphically polyhedral, or more generally \(c_0\)-saturated spaces,
under quotient maps. It was asked in \cite{r2} if the dual of a separable isomorphically polyhedral
space is \(\ell_1\)-saturated that is, every closed, linear, infinite-dimensional subspace of the dual
contains a further subspace isomorphic to \(\ell_1\). It is an open problem (\cite{o}, \cite{r2}, \cite{r3})
if every quotient of a \(C(K)\) space with \(K\) a countable and compact metric space, 
is \(c_0\)-saturated. It was shown in \cite{g} that for every \(p \in (1, \infty)\),
\(\ell_p\) is isomorphic to a quotient of an isomorphically polyhedral space with an unconditional basis.
The purpose of this article is to extend this result by showing the following
\begin{Thm} \label{main}
Let \(Z\) be a Banach space with a shrinking, unconditional basis satisfying an upper \(p\)-estimate for
some \(p > 1\). Then, there exists an isomorphically polyhedral space with an unconditional basis which
admits a quotient isomorphic to \(Z\).
\end{Thm}
We obtain in particular, that reflexive spaces with unconditional bases and non-trivial type, are isomorphic
to quotients of isomorphically polyhedral spaces with unconditional bases. The same property holds
for \((\sum c_0)_{\ell_p}\), for all \(p \in (1, \infty)\). Using the fact that Tsirelson's space \(T\)
\cite{fj} is isomorphic to its modified version (\cite{j2}, \cite{co}), we also obtain that 
Tsirelson's original space \(T^*\) \cite{t} is isomorphic 
to a quotient of an isomorphically polyhedral space with an unconditional basis.

A consequence of Theorem \ref{main} is the existence of new examples of \(c_0\)-saturated spaces admitting
reflexive quotients. The first example of such a space was given in \cite{ckt}, where a certain Orlicz function
space was shown to admit \(\ell_2\) as a quotient. It was proved in \cite{l} that \(\ell_2\) is a quotient of
a \(c_0\)-saturated space with an unconditional basis.

More general results were obtained in \cite{af} with the use of interpolation methods.
They showed that every reflexive space with an unconditional
basis has a block subspace which is isomorphic to a quotient of a \(c_0\)-saturated space.
This result has been recently extended to cover all separable reflexive spaces. It is shown in
\cite{ar} that every such space is a quotient of a \(c_0\)-saturated space with a basis.
\section{Preliminaries}
Our notation is standard as may be found in \cite{lt}. We shall consider Banach spaces 
over the real field. If \(X\) is a Banach space then \(B_X\) stands for its closed unit ball. 
A bounded subset \(B\) of the dual \(X^*\) of \(X\) is {\em norming}, if there exists 
\(\rho > 0\) such that \(\sup_{x^* \in B} |x^*(x)| \geq \rho \|x\|\), for all \(x \in X\).
In case \(B \subset B_{X^*}\) and \(\rho =1\), \(B\) is said to isometrically norm \(X\).

\(X\) is said to contain an {\em isomorph} of the Banach space \(Y\) (or, equivalently, that
\(X\) contains \(Y\) isomorphically), if there exists a bounded linear injection from
\(Y\) into \(X\) having closed range. 

A sequence \((x_n)\) in a Banach space is {\em semi-normalized} if \(\inf_n \|x_n \| > 0 \) and
\(\sup_n \|x_n \| < \infty\). It is called 
a {\em basic} sequence provided it is a Schauder basis for its closed linear
span in \(X\). A Schauder basis \((x_n)\) for the space \(X\) is {\em shrinking}, if
the sequence of functionals \((x_n^*)\), biorthogonal to \((x_n)\), is a Schauder basis
for \(X^*\).

If \((x_n)\) and \((y_n)\) are basic sequences, then 
\((x_n)\) {\em dominates} \((y_n)\) if there is a constant \(C  > 0\) so that  
\(\|\sum_{i=1}^n a_i y_i \| \leq  C \|\sum_{i=1}^n a_i x_i \| \),
for every choice of scalars \((a_i)_{i=1}^n\) and all \(n \in \mathbb{N}\). 
The basic sequences \((x_n)\) and \((y_n)\) are {\em equivalent}, if each one of them dominates 
the other.
A basic sequence \((x_n)\) is called {\em suppression} \(1\)-unconditional, if
\(\|\sum_{i \in F} a_i x_i \|\) \(\leq \|\sum_{i=1}^n a_i x_i\|\), for all \(n \in \mathbb{N}\),
all choices of scalars \((a_i)_{i=1}^n\), and every \(F \subset \{1, \dots , n\}\).
Evidently, such a basic sequence is unconditional, that is 
every series of the form \(\sum_n a_n x_n \) converges unconditionally, whenever it converges.

If \((x_n)\) is a basic sequence in some Banach space \(X\), then a sequence \((u_n)\)
of non-zero vectors in \(X\),
is a {\em block basis} of \((x_n)\) if there 
exist a sequence of non-zero scalars \((a_n)\) and a sequence \((F_n)\) of sucessive finite subsets
of \(\mathbb{N}\) (i.e, \(\max F_n < \min F_{n+1}\) for all \(n \in \mathbb{N}\)), so that
\(u_n = \sum_{i \in F_n} a_i x_i\), for all \(n \in \mathbb{N}\). We then call \(F_n\) the
{\em support} of \(u_n\) for all \(n \in \mathbb{N}\). Any member of a block basis of
\((x_n)\) will be called a {\em block} of \((x_n)\).

A Banach space with an unconditional basis \((e_n)\) satisfies an {\em upper } \(p\)-{\em estimate}, for
some \(p > 1\), if there exists a constant \(C > 0\) so that
\(\|\sum_{i=1}^n u_i \| \leq C (\sum_{i=1}^n \|u_i \|^p)^{1/p}\), for every choice
\((u_i)_{i=1}^n\) of disjointly supported blocks of \((e_n)\).

Given finite subsets \(E\), \(F\) of \(\mathbb{N}\), then the notation
\(E < F\) indicates that \(\max E < \min F\). If \(\mu \), \(\nu\) are finitely supported
signed measures on \(\mathbb{N}\), then we write \(\mu < \nu\) if 
\(\mathrm{ supp } \, \mu < \mathrm{ supp } \, \nu\).

A family \(\mathcal{F}\) of finite subsets of \(\mathbb{N}\) is said to be
{\em compact} if it is compact in the topology of pointwise convergence
in \(2^{\mathbb{N}}\). We next recall the Schreier family
\[S_1 = \{F \subset \mathbb{N}: \, |F| \leq \min F \} \cup \{\emptyset\}.\]
The higher ordinal Schreier families \(\{S_\alpha: \, \alpha < \omega_1\}\), were introduced in \cite{aa}
where it is shown that
\(S_\alpha\) is homeomorphic to the ordinal interval
\([1, \omega^{\omega^\alpha}]\), for all \(\alpha < \omega_1\). 
\section{The main construction} \label{S1}
Let \(Z\) be a Banach space with a normalized, shrinking, unconditional 
basis \((z_n)\). By renorming if necessary, we may assume that
\(\|\sum_n a_n z_n \| = \|\sum_n |a_n| z_n \|\), for every \((a_n) \in c_{00}\).
Define \(\phi_Z \colon \mathbb{N} \to \mathbb{R}\), by
\begin{align}
\phi_Z(k) = \sup &\biggl \{ \biggl \|\sum_{i=1}^k u_i \biggr \|: \, (u_i)_{i=1}^k \text{ are } 
\text{ finitely and disjointly supported} \notag \\
&\text{ blocks of } (z_n), \,
\|u_i \| \leq 1, \, \forall \, i \leq k \biggr \}, \, \forall \, k \in \mathbb{N}. 
\notag
\end{align}
It is easy to see that \(\phi_Z\) is a submultiplicative function, that is
\(\phi_Z(mn) \leq \phi_Z(m) \phi_Z(n)\) for all integers \(m\), \(n\).
It follows from the results of \cite{j} (cf. also Theorem 1.f.12 in
\cite{lt}), that \(Z\) satisfies an upper \(p\)-estimate for some \(p > 1\) if, and only if,
\(\phi_Z(k) < k \) for some \(k \geq 2\).
To define the desired isomorphically polyhedral space, we need to introduce some notations and definitions.
Let \(Z\) satisfy an upper-\(p\) estimate for some \(p > 1\).
For simplicity, we shall write \(\phi\) instead of \(\phi_Z\).

{\bf Notation}. Fix some \(k_0 \geq 2\) with \(\phi(k_0) < k_0\) and choose
\(\lambda \in (\phi(k_0)/k_0, 1)\). Set \(\epsilon_n = (1/k_0)^n\) and
\(\delta_n = \lambda^n\), for all \(n \in \mathbb{N} \cup \{0\}\).

We next choose a sequence \((F_n)\) of successive finite subsets of \(\mathbb{N}\) (i.e., 
\(\max F_n < \min F_{n+1}\) for all \(n \in \mathbb{N}\)) so that
\(\sum_n (1/|F_n|) < 1\) and \(1 + (1/\delta_n) < \epsilon_{n+1} \min F_{n +1}\), for
all \(n \in \mathbb{N} \cup \{0\}\).
 
{\bf Notation}. Given \(n \in \mathbb{N}\), let \(e_n^*\) denote the point mass measure at \(n\).
Let \(\mathcal{P}\) denote the set of non-negative, finitely supported measures \(\mu\) on \(\mathbb{N}\)
of the form \(\mu = \sum_i \lambda_i e_{j_i}^*\), where \(\lambda_i \in [0,1]\) and
\(j_i \in F_i\) for all \(i \in \mathbb{N}\). 

Let \(\mathcal{P}_1 = \{\mu \in \mathcal{P}: \, \sum_i \mu(F_i) \leq 1\}\).
\begin{Def}
A measure \(\mu \in \mathcal{P}\) is said to be \(Z\)-bounded, provided that
\(\sum_i \mu(G_i) \leq 1\) for every sequence \((G_i)\) of subsets of \(\mathbb{N}\)
with \(G_i\) an initial segment of \(F_i\) (we allow \(G_i = \emptyset\)) for
all \(i \in \mathbb{N}\), such that \(\|\sum_i (|G_i|/|F_i|) z_i \|_Z \leq 1\). 
\end{Def} 
\begin{remark} \label{r1}
A typical example of a \(Z\)-bounded measure is as follows: Let \((\rho_i)\) be a sequence
of non-negative scalars such that \(\|\sum_i \rho_i z_i^* \|_{Z^*} \leq 1\). For every
\(i \in \mathbb{N}\), choose an initial segment \(G_i\) of \(F_i\). Let \(I\) be a finite
subset of \(\mathbb{N}\) and define
\(\mu = \sum_{i \in I} \rho_i (|G_i|/|F_i|)e_{\max G_i}^* \).
Then \(\mu\) is \(Z\)-bounded. Indeed, let \((H_i)\) be a sequence of finite subsets of 
\(\mathbb{N}\) with each \(H_i\) being an initial segment of \(F_i\), such that
\(\|\sum_i (|H_i|/|F_i|) z_i \|_Z \leq 1\). Let \(i \in I\). Then, either \(G_i\)
is an initial segment of \(H_i\), or, \(\max G_i > \max H_i\). If the former, then
\(\mu(H_i) = \rho_i (|G_i|/|F_i|)\) and \(|G_i| / |F_i| \leq |H_i| /|F_i|\).
If the latter, then \(\mu(H_i) = 0\). Let \(I_1\) be the subset of \(I\) consisting
of those elements of \(I\) for which the first alternative occurs. Then, 
\(\|\sum_{i \in I_1} (|G_i|/|F_i|) z_i \|_Z \leq 1\), by our initial assumptions on 
\((z_i)\), and so 
\[\sum_i \mu(H_i) = \sum_{i \in I_1} \mu(H_i) = \sum_{i \in I_1} \rho_i (|G_i|/|F_i|) \leq 1.\]
\end{remark} 
\begin{Def}
A finite sequence \((\mu_i)_{i=1}^k\) of non-zero, disjointly supported members of
\(\mathcal{P}_1\) is called admissible, if for every \(n \in \mathbb{N}\) we have that
\(F_n \cap \mathrm{ supp } \, \mu_i \ne \emptyset\) for at most one \(i \leq k\), and,
moreover, if \(F_n \cap \mathrm{ supp } \, \mu_i \ne \emptyset\) for some \(n \in \mathbb{N}\) and
\(i \leq k\), then \(k \leq \min F_n\).
\end{Def}
Note in particular that \(\{ \min \mathrm{ supp } \, \mu_i : \, i \leq k \} \in S_1\) if 
\((\mu_i)_{i=1}^k \) is admissible. We can now describe a norming subset of the space we wish to construct.
\begin{align}
\mathcal{M} &= \biggl \{ \mu \in \mathcal{P}, \mu \text{ is } Z-\text{ bounded}, \, 
\mu = \sum_{i=1}^k \mu_i,  \, k \in \mathbb{N} 
\text{ and } \notag \\
&(\mu_i)_{i=1}^k \text{ is an admissible sequence in } \mathcal{P}_1 
\biggr \} 
\cup \{e_n^* : \, n \in \mathbb{N}\} \cup \{0\}. \notag
\end{align}
It is easy to see that \(\mu | I \in \mathcal{M}\) for every \(\mu \in \mathcal{M}\)
and all \(I \subset \mathbb{N}\). We can now define a norm \(\| \cdot \|_{\mathcal{M}}\)
on \(c_{00}\) by 
\[\|x\|_{\mathcal{M}} = \max \biggl \{ \biggl | \sum_i \mu(\{i\}) x(i) \biggr |: \, \mu \in \mathcal{M}
\biggr \}, \forall \, x \in c_{00}.\]
Let \(X_{\mathcal{M}}\) be the completion of \((c_{00}, \| \cdot \|_{\mathcal{M}})\).
Since \(\mathcal{M}\) is closed under restrictions to subsets of \(\mathbb{N}\),
we obtain that the natural basis \((e_n)\) of \(c_{00}\) becomes a normalized, suppression \(1\)-unconditional basis for 
\(X_{\mathcal{M}}\). Note also that \(\mathcal{M}\) is an isometrically norming subset of
\(B_{X_{\mathcal{M}}^*}\).
Our objective is to show that \(X_{\mathcal{M}}\) is isomorphically polyhedral and admits
a quotient isomorphic to \(Z\). The first task is accomplished through the next result, proved in \cite{g} via 
Elton's theorem \cite{e}.
\begin{Prop} \label{P1}
Let \(X\) be a Banach space with a normalized Schauder basis \((e_n)\). Let
\((e_n^*)\) denote the sequence of functionals biorthogonal to \((e_n)\).
Assume there is a bounded norming subset \(B\) of \(X^*\) with the following property:
There exists a compact family \(\mathcal{F}\) of finite subsets
of \(\mathbb{N}\) such that for every \(b^* \in B\) there exist \(F \in \mathcal{F}\)
and a finite sequence \((b_k^*)_{k \in F}\) of finitely supported absolutely 
sub-convex combinations of \((e_n^*)\) so that
\(b^* = \sum_{k \in F} b_k^*\) and
\(\min \mathrm{ supp } \, b_k^* \geq k \) for all \(k \in F\).
Then, \(\sum_n |b^*(e_n)| < \infty \), for all \(b^* \in \overline{B}^{w^*}\) and 
\(X\) is isomorphically polyhedral.
\end{Prop} 
\begin{Cor} \label{C1}
\(X_{\mathcal{M}}\) is isomorphically polyhedral and \((e_n)\) is an unconditional, shrinking normalized basis for 
\(X_{\mathcal{M}}\).
\end{Cor}
\begin{proof}
The fact that \((e_n)\) is normalized and unconditional follows directly from the definition of
\(\mathcal{M}\).
Next, let \( \mu \in \mathcal{M}\). We verify that the conditions given in Proposition \ref{P1}
are fulfilled by \(\mu\) with \(\mathcal{F} = S_1\). Indeed, this is obvious when \(\mu = e_n^*\)
for some \(n \in \mathbb{N}\). Otherwise, \(\mu = \sum_{i=1}^k \mu_i \) for some \(k \in \mathbb{N}\) 
and an admissible family \((\mu_i)_{i=1}^k\) of members of \(\mathcal{P}_1\).
Since every \(\mu_i\) is a finitely supported sub-convex combination of \((e_n^*)\), and
\(\{\min \mathrm{ supp } \, \mu_i : \, i \leq k \}
\in S_1\),
the first assertion of the corollary follows from Proposition \ref{P1}.
Since \(X_{\mathcal{M}}\) is \(c_0\)-saturated, it can not contain any isomorph of \(\ell_1\) 
and thus a classical result due to James yields that \((e_n)\) is
shrinking.
\end{proof}  
In the sequel, we shall write \(\| \cdot \|\), resp. \(\| \cdot \|_*\), 
instead of \(\| \cdot \|_{\mathcal{M}}\), resp. \(\| \cdot \|_{X_{\mathcal{M}}^*}\).

The next lemma describes a simple method for selecting subsequences of \((e_n)\), equivalent to
the \(c_0\)-basis. 
\begin{Lem} \label{L1}
Let \(I\) be a finite subset of \(\mathbb{N}\) and
\((G_n)_{n \in I}\) a sequence of finite subsets of \(\mathbb{N}\) with \(G_n\) an initial
segment of \(F_n\) for all \(n \in I \), such that \(\|\sum_{n \in I} (|G_n|/|F_n|)z_n \|_Z \leq 1\).
Then, \(\| \sum_{n \in I} \sum_{k \in G_n} e_k \| \leq 1\).
\end{Lem}
\begin{proof}
Set \(u = \sum_{n \in I} \sum_{k \in G_n} e_k \). We show that \(\mu (u) \leq 1\)
for all \(\mu \in \mathcal{M}\). In case \(\mu = e_n^*\)
for some \(n \in \mathbb{N}\), then the assertion trivially holds.
Every other element \(\mu \in \mathcal{M}\) is \(Z\)-bounded
and so \(\mu (u) = \sum_{n \in I} \mu(G_n) \leq 1\), as \(\|\sum_{n \in I} (|G_n|/|F_n|)z_n \|_Z \leq 1\).
\end{proof}
\section{ \(Z\) is isomorphic to a quotient of \(X_{\mathcal{M}}\)} \label{S2}
The main result of this section is the following
\begin{Thm} \label{T3}
Let \(u_n^* = \sum_{i \in F_n} (1/|F_i|) e_i^* \), for all \(n \in \mathbb{N}\).
Then \((u_n^*)\) is equivalent to \((z_n^*)\), the sequence of functionals biorthogonal to \((z_n)\).
\end{Thm}
The proof of this result will follow after a
series of lemmas, where we first show that \((u_n^*)\) dominates \((z_n^*)\) and then that it
is dominated by \((z_n^*)\). Note that our initial assumptions on \((z_n)\) yield that \((z_n^*)\)
is a normalized, suppression \(1\)-unconditional basis for \(Z^*\).  
\begin{Lem} \label{L2}
\((u_n^*)\) dominates \((z_n^*)\).
\end{Lem}
\begin{proof}
Note first that \((u_n^*)\) is a normalized (convex) block basis of \((e_n^*)\) in \(X_{\mathcal{M}}^*\).
This is so since \(\|\sum_{i \in F_n } e_i \| = 1\), for all \(n \in \mathbb{N}\) (observe that the support of every
member of \(\mathcal{M}\) meets each \(F_n\) in at most one point).   
We thus obtain that \((u_n^*)\) is normalized and suppression \(1\)-unconditional.

Now let \(n \in \mathbb{N}\) and \((a_i)_{i=1}^n\) be scalars in \([0,1]\) with 
\(\|\sum_{i=1}^n a_i z_i^* \|_{Z^*} \leq 1\). We next find scalars \((b_i)_{i=1}^n\) in \([0,1]\)
so that \(\|\sum_{i=1}^n b_i z_i \|_Z \leq 1\) and \(\sum_{i=1}^n a_i b_i = \|\sum_{i=1}^n a_i z_i^* \|_{Z^*}\). 

For each \(i \leq n\) choose an initial segment \(G_i\) of \(F_i\) so that
\[|G_i|/|F_i| \leq b_i < (|G_i| / |F_i|) + (1/|F_i|).\]
This choice ensures, thanks to our initial assumptions on \((z_n)\), that 
\[\biggl \| \sum_{i=1}^n (|G_i|/|F_i|) z_i \biggr \|_Z \leq \biggl \| \sum_{i=1}^n b_i z_i \biggr \|_Z \leq 1.\]
Let \(u = \sum_{i=1}^n \sum_{k \in G_i} e_k \). We deduce from Lemma \ref{L1}, that \(\|u\| \leq 1\).
It follows now that 
\begin{align}
\biggl \|\sum_{i=1}^n a_i u_i^*  \biggr \|_* &\geq \sum_{i=1}^n a_i u_i^*(u) 
= \sum_{i=1}^n a_i \sum_{k \in F_i} (1/|F_i|) e_k^*(u) = \sum_{i=1}^n a_i (|G_i|/|F_i|) \notag \\
&\geq \sum_{i=1}^n a_i (b_i - (1/|F_i|)) \geq \biggl \| \sum_{i=1}^n a_i z_i^* \biggr \|_{Z^*}
 - \sum_{i=1}^n (1/|F_i|). \notag
\end{align}
Next suppose that \(n \in \mathbb{N}\) and \((a_i)_{i=1}^n\) is a scalar sequence
satisfying 

\noindent \(\|\sum_{i=1}^n a_i z_i \|_{Z^*} =1\). Let \(I^{+} = \{i \leq n : \, a_i \geq 0 \}\)
and \(I^{-} = I \setminus I^{+}\). Our preceding work yields that
\[ \biggl \|\sum_{i \in I^j} a_i u_i^*  \biggr \|_* \geq  \biggl \| \sum_{i \in I^j} a_i z_i^* \biggr \|_{Z^*}-
\sum_{i \in I^j} (1/|F_i|), \, \forall \, j \in \{+, -\}.\]
We deduce now from the above and the fact that \((u_n^*)\) is suppression \(1\)-unconditional, that
\[2 \biggl \|\sum_{i=1}^n a_i u_i^*  \biggr \|_* \geq 1 - \sum_{i=1}^\infty (1/|F_i|) > 0. \]
Therefore, letting \(A = (1/2)(1 - \sum_{i=1}^\infty (1/|F_i|)) > 0\), we obtain that
\[\biggl \|\sum_{i=1}^n a_i u_i^* \biggr \|_* \geq A 
\biggl \|\sum_{i=1}^n a_i z_i^*  \biggr \|_{Z^*}\] 
for every \(n \in \mathbb{N}\) 
and all choices of scalars \((a_i)_{i=1}^n \subset \mathbb{R}\). The proof of the lemma is now complete.
\end{proof}  
\begin{Lem} \label{L3}
Let \(u = \sum_i a_i e_i \) be a finitely supported vector in \(X_{\mathcal{M}}\), with \(\|u \| \leq 1\)
and \(a_i \geq 0\), for all \(i \in \mathbb{N}\). 
Let \(\mu \in \mathcal{P}\) be \(Z\)-bounded and write
\(\mu = \sum_{i \in I} \lambda_i e_{j_i}^* \), where \(I\) is a finite subset of \(\mathbb{N}\),
\(j_i \in F_i\) and \(\lambda_i \in (0, 1]\) for all \(i \in I\). 
Suppose that there exists \(n \in \mathbb{N} \cup \{0\}\) with \(n < \min I\)
such that \(a_{j_i} \geq \epsilon_{n+1} \), for all \(i \in I\).
Then, \(\|\mu\|_* \leq 2\).
\end{Lem}
\begin{proof}
Set \(I_1 = \{i \in I : \, \lambda_i \geq 1/2\}\) and \(I_2 = I \setminus I_1\).
The assertion of the lemma will follow once we show that 
\(\mu | I_s \in \mathcal{M}\)
for every \(s \leq 2\).
To this end, we first claim that \(|I_1| \leq 2/\epsilon_{n+1}\).

Indeed, if the claim were false, let \(d = 2/ \epsilon_{n +1}\)
and choose \(d+1\) elements \(i_1 < \dots < i_{d+1}\) in \(I_1\). 
Since \(d < \min F_i\), for all \(i \in I\) (by the initial choice of 
the sequence \((F_i)\)) we have that \((\lambda_s e_{j_{i_s}}^*)_{s=1}^{d +1}\)
is an admissible sequence of members of \(\mathcal{P}_1\). Moreover, since
\(\mu\) is \(Z\)-bounded, so is \(\mu | J \), for every \(J \subset \mathbb{N}\).
Therefore, \(\sum_{s=1}^{d+1} \lambda_s e_{j_{i_s}}^* \in \mathcal{M}\) and so
our assumptions on \(u\) yield that
\[1 \geq \|u\| \geq  \sum_{s=1}^{d+1} \lambda_s e_{j_{i_s}}^* (u) =
\sum_{s=1}^{d+1} \lambda_s a_{j_{i_s}} \geq (d+1)( \epsilon_{n + 1} /2) > 1, \]
a contradiction that proves our claim.
It follows now that \((\lambda_i e_{j_i}^*)_{i \in I_1}\) is an admissible family
and hence \(\mu | I_1 \in \mathcal{M}\).

We next show that \(\mu | I_2 \in \mathcal{M}\).
We first choose a non-empty initial segment
\(J_1\) of \(I_2\) which is maximal with respect to the condition
\(\sum_{i \in J_1} \lambda_i \leq 1\). In case \(J_1 = I_2\), 
the assertion follows as \(\mathcal{P}_1 \subset \mathcal{M}\).

If \(J_1\) is a proper initial segment of \(I_2\), then, by maximality, we must have
\[1/2 < \sum_{i \in J_1} \lambda_i \leq 1, \]
as \( \lambda_i < 1/2\), for every \(i \in I_2\).
We now set \(\mu_1 = \sum_{i \in J_1} \lambda_i e_{j_i}^* \). 
This measure belongs to \(\mathcal{P}_1\) and satisfies \(\mu_1(u) > \epsilon_{n +1} /2\) because
\(a_{j_i} \geq \epsilon_{n+1} \), for all \(i \in I\). 

We repeat the same process to \(I_2 \setminus J_1\) and obtain a non-empty initial segment \(J_2\)
of \(I_2 \setminus J_1\), and a measure \(\mu_2 = \sum_{i \in J_2} \lambda_i e_{j_i}^* \) in
\(\mathcal{P}_1\)
so that either \(J_1 \cup J_2 = I_2\), or, \(J_2\) is a proper initial segment of
\(I_2 \setminus J_1\) satisfying \(\mu_2(u) > \epsilon_{n +1} /2\). 
If the former, the process stops. If the latter, the process continues.
Because \(I_2\) is finite, this process will terminate after a finite number of steps, say \(k\).
We shall then have produced successive subintervals \(J_1 < \dots < J_k \) of \(I_2\) 
with \(I_2 = \cup_{r=1}^k J_r\), and
measures \(\mu_1 < \dots < \mu_k \) in \(\mathcal{P}_1\) with
\(\mu_r = \sum_{i \in J_r} \lambda_i e_{j_i}^* \), for all \(r \leq k\).
Moreover, \(\mu_r(u) >  \epsilon_{n+1} /2\), for all \(r < k\).

We claim that \(k \leq d = 2 /\epsilon_{n +1}\). 
Indeed, assuming the contrary, we have by the choice of \(k\), that
\(\mu_r(u) \geq \epsilon_{n +1} /2\), for all \(r \leq d \).
But also, \(2/\epsilon_{n+1} < \min F_i\), as \(i \geq n+1\) for all \(i \in I\), and thus,
\((\mu_r)_{r=1}^d\) is admissible. Since \(\nu = \sum_{r=1}^d \mu_r = \mu | \cup_{r=1}^d J_r\),
it is \(Z\)-bounded and so \(\nu \in \mathcal{M}\).
Therefore,
\[1 \geq \|u\| \geq \sum_{r=1}^d \mu_r (u) > d \epsilon_{n+1} /2 = 1,\]
which is a contradiction. Hence, \(k \leq 2/ \epsilon_{n+1} \leq \min F_i\), for all
\(i \in I\). This implies that \((\mu_r)_{r=1}^k\) is admissible and so
\(\mu | I_2 = \sum_{r=1}^k \mu_r \in \mathcal{M}\), completing the proof of
the lemma.
\end{proof}
\begin{Lem} \label{L4}
Let \(u = \sum_i a_i e_i \) be a finitely supported vector in \(X_{\mathcal{M}}\), with \(\|u \| \leq 1\)
and \(a_i \geq 0\), for all \(i \in \mathbb{N}\). Let \(I\) be a finite subset of \(\mathbb{N}\) and
\(n \in \mathbb{N} \cup \{0\}\) with \(n < \min I\). Suppose that for each \(i \in I\) there exists
an initial segment \(G_i\) of \(F_i\) such that \(a_{\max G_i} \geq \epsilon_{n+1}\). Let
\((\rho_i) \) be a sequence of non-negative scalars such that 
\(\|\sum_i \rho_i z_i^* \|_{Z^*} \leq 1\). Assume further that there exists a family \(\mathcal{J}\)
of pairwise disjoint subsets of \(I\) such that every member \(J\) of \(\mathcal{J}\)
satisfies the following conditions:
\((1)\) \(\|\sum_{i \in J} (|G_i|/|F_i|) z_i \|_Z \leq 1\),
and \((2)\) \(\sum_{i \in J} \rho_i (|G_i|/|F_i|) \geq \delta_n \).
Then, \(|\mathcal{J}| < (1/\epsilon_{n +1 } ) (1 + [1/\delta_n])\) and,
\[\sum_{J \in \mathcal{J}} \sum_{i \in J} \rho_i (|G_i|/|F_i|) \leq 
(1 + 1/\delta_n )\phi(k_0)^{n+1}.\]
\end{Lem}
\begin{proof}
We set \(d = (1/\epsilon_{n +1 } ) (1 + [1/\delta_n])\) and assume, to the contrary,
that \(|\mathcal{J}| \geq d\).
We can now select pairwise disjoint members \(J_1, \dots, J_d\) of \(\mathcal{J}\)
and define
\(\mu_r = \sum_{i \in J_r} \rho_i (|G_i|/|F_i|)e_{\max G_i}^* \), for all \(r \leq d\).
Since \(\|\sum_i \rho_i z_i^* \|_{Z^*} \leq 1\), \((1)\) implies that \(\mu_r \in \mathcal{P}_1\)
for all \(r \leq d\). But also, \(n +1 \leq i\), for all \(i \in I\) and thus,
\(d < \min F_i\) for all \(i \in I\), because of our initial assumptions on \((F_i)\).
It follows now that \((\mu_r)_{r=1}^d\) is admissible. Let \(\mu = \sum_{r=1}^d \mu_r\).
Since \(\mu\) is \(Z\)-bounded (see Remark \ref{r1}), we infer from the above that
\(\mu \in \mathcal{M}\) and therefore,
\begin{align}
1 \geq \mu(u) &= \sum_{r=1}^d \sum_{i \in J_r} \rho_i (|G_i|/|F_i|) a_{\max G_i} \notag \\
&\geq \sum_{r=1}^d \sum_{i \in J_r} \rho_i (|G_i|/|F_i|) \epsilon_{n+1} \geq
\epsilon_{n+1} \delta_n d, \text{ by } (2). \notag
\end{align}
Hence, \(d \leq 1/ (\epsilon_{n+1} \delta_n) < (1/\epsilon_{n +1 } ) (1 + [1/\delta_n]) = d\).
This contradiction shows that \(|\mathcal{J}| < d\), as required.

We next verify the second assertion of the lemma. To this end,
\begin{align}
\sum_{J \in \mathcal{J}} \sum_{i \in J} \rho_i (|G_i|/|F_i|) &=
\biggl ( \sum_i \rho_i z_i^* \biggr ) \biggl (
\sum_{J \in \mathcal{J}} \sum_{i \in J} (|G_i|/|F_i|) z_i \biggr ) \notag \\
&\leq \biggl 
\|\sum_{J \in \mathcal{J}} \sum_{i \in J}  (|G_i|/|F_i|) z_i \biggr \|_Z \leq \phi(|\mathcal{J}|),
\text{ by } (1), \notag \\
&\leq \phi \bigl ( (1/\epsilon_{n+1})(1 + [1/\delta_n]) \bigr ) \leq
\phi(1/\epsilon_{n+1}) \phi(1 + [1/\delta_n]), \notag \\
&\text{ by the submultiplicativity of } \phi, \notag \\
&\leq (1 + 1/ \delta_n) \phi(k_0^{n+1}) \leq (1 + 1/\delta_n) \phi(k_0)^{n+1}, \notag
\end{align}
using once again the fact that \(\phi\) is submultiplicative. This concludes the proof
of the lemma.  
\end{proof}
\begin{Lem} \label{L5}
There exists a constant \(C > 0\) such that for every \(I \subset \mathbb{N}\), finite,
and every collection of scalars \((\rho_i)_{i \in I}\) with
\(\|\sum_i \rho_i z_i^* \|_{Z^*} \leq 1\), we have that
\(\|\sum_{i \in I} \rho_i u_i^* \|_* \leq C\).
Consequently, \((z_i^*)\) dominates \((u_i^*)\).
\end{Lem}
\begin{proof}
The unconditionality of \((z_i^*)\) clearly allows us establish the assertion of the lemma
under the additional assumption that \(\rho_i \geq 0\), for all \(i \in I\).
Given \(n \in \mathbb{N} \cup \{0\}\), let \(I_n = \{i \in I : \, i > n\}\).
Let \(u = \sum_i a_i e_i \) be a finitely supported vector in \(X_{\mathcal{M}}\), with \(\|u \| \leq 1\)
and \(a_i \geq 0\), for all \(i \in \mathbb{N}\). We have the following initial estimate
\begin{align} \label{E1}
\sum_{i \in I}\rho_i u_i^*(u) &= \sum_{i \in I} \rho_i \sum_{j \in F_i} a_j /|F_i| =
\sum_{n=0}^\infty \sum_{i \in I} \rho_i \sum_{j \in F_i : \, a_j \in (\epsilon_{n+1}, \epsilon_n]} a_j /|F_i| 
\\
&\leq \sum_{n=0}^\infty n \epsilon_n + \sum_{n=0}^\infty  \sum_{i \in I_n} \rho_i
\sum_{j \in F_i : \, a_j \in (\epsilon_{n+1}, \epsilon_n]} a_j /|F_i|. \notag
\end{align}
Fix some \(n \in \mathbb{N} \cup \{0\}\). For each \(i \in I_n\), let \(j_i^n\)
denote the largest \(j \in F_i\) so that \(a_j \in (\epsilon_{n +1}, \epsilon_n]\).
In case \(j_i^n\) fails to exist, then the corresponding summand in \eqref{E1} equals \(0\)
and thus it has no effect into our estimates. Let \(G_i^n \) be the initial segment of
\(F_i\) with \(\max G_i^n = j_i^n \). It is clear that
\begin{equation} \label{E2}
\sum_{j \in F_i : \, a_j \in (\epsilon_{n+1}, \epsilon_n]} a_j /|F_i| \leq 
(\epsilon_n / \epsilon_{n+1}) a_{j_i^n} (|G_i^n|/|F_i|), \, \forall \, i \in I_n.
\end{equation}
Call a subset \(J\) of \(I_n\) bad, if 
\[ \biggl \| \sum_{i \in J} (|G_i^n|/|F_i|) z_i \biggr \|_Z \leq 1, \text{ and }
\sum_{i \in J} \rho_i (|G_i^n|/|F_i|) \geq \delta_n .\]
It is clear that we can extract a maximal, under inclusion, family \(\mathcal{J}_n\) 
consisting of pairwise disjoint, bad subsets of \(I_n\). We now observe the following:
If \(J \subset I_n \setminus \cup \mathcal{J}_n \) and
\(\| \sum_{i \in J} (|G_i^n|/|F_i|) z_i \biggr \|_Z \leq 1\), 
then \(\sum_{i \in J} \rho_i (|G_i^n|/|F_i|) < \delta_n \). Indeed, if that were not so,
then \(J\) would be bad, contradicting the maximality of \(\mathcal{J}_n\).

Letting \(\mu_n = \sum_{i \in I_n \setminus \cup \mathcal{J}_n} \rho_i (|G_i^n|/|F_i|) e_{j_i^n}^*\),
we infer from the preceding observation, that \((1/\delta_n) \mu_n \) is a \(Z\)-bounded measure.
On the other hand, \(a_{j_i^n} \geq \epsilon_{n+1}\), for each \(i \in I_n\) for which
\(j_i^n\) exists, and therefore Lemma \ref{L3} yields that
\(\|\mu_n \|_* \leq 2 \delta_n\). Taking in account \eqref{E2}, we obtain the estimate
\begin{align} \label{E3}
\sum_{i \in I_n \setminus \cup \mathcal{J}_n} \rho_i 
\sum_{j \in F_i : \, a_j \in (\epsilon_{n+1}, \epsilon_n]} a_j /|F_i| 
&\leq (\epsilon_n / \epsilon_{n+1}) \sum_{i \in I_n \setminus \cup \mathcal{J}_n} \rho_i 
a_{j_i^n} (|G_i^n|/|F_i|)  \\
&= (\epsilon_n / \epsilon_{n+1}) \mu_n(u) \leq 2k_0 \delta_n .
\notag
\end{align}
We next employ Lemma \ref{L4} to obtain the estimate
\begin{equation} \label{E4}
\sum_{J \in \mathcal{J}_n } \sum_{i \in J} \rho_i (|G_i^n| / |F_i|) \leq
(1 + 1/ \delta_n) \phi(k_0)^{n+1}
\end{equation}
and thus, taking \eqref{E2} into account, we reach the estimate
\begin{align} \label{E5}
&\sum_{i \in \cup \mathcal{J}_n} \rho_i 
\sum_{j \in F_i : \, a_j \in (\epsilon_{n+1}, \epsilon_n]} a_j /|F_i| \\
&\leq (\epsilon_n / \epsilon_{n+1}) \sum_{i \in \cup \mathcal{J}_n} \rho_i 
a_{j_i^n} (|G_i^n|/|F_i|) 
= k_0 \sum_{J \in \mathcal{J}_n} \sum_{i \in J} \rho_i a_{j_i^n} (|G_i^n|/|F_i|) \notag \\
&\leq k_0 \epsilon_n \sum_{J \in \mathcal{J}_n } \sum_{i \in J} \rho_i (|G_i^n| / |F_i|) 
\leq k_0 \epsilon_n (1 + 1/ \delta_n) \phi(k_0)^{n+1}, \text{ by } \eqref{E4}, \notag \\
&= k_0 \phi(k_0) (1 + 1/ \delta_n) (\phi(k_0) / k_0)^n 
= k_0 \phi(k_0) [ (\phi(k_0) / k_0)^n + (\phi(k_0) / \lambda k_0)^n]. \notag
\end{align} 
Let \(B_n =  k_0 \phi(k_0) [ (\phi(k_0) / k_0)^n + (\phi(k_0) / \lambda k_0)^n]\).
Equations \eqref{E3} and \eqref{E5} now yield that
\[\sum_{i \in I_n} \rho_i \sum_{j \in F_i : \, a_j \in (\epsilon_{n+1}, \epsilon_n]} a_j /|F_i| 
\leq 2k_0 \delta_n + B_n, \, \forall \, n \geq 0,\]
and so, finally, \eqref{E1} gives us the estimate
\[\sum_{i \in I} \rho_i u_i^*(u) \leq \sum_{n=0}^\infty ( n \epsilon_n + 2k_0 \delta_n + B_n) < \infty,\]
as \(0 < \phi(k_0) / k_0 < \lambda < 1\).
The assertion of the lemma now follows as \((e_n)\) is an unconditional basis for \(X_{\mathcal{M}}\).
\end{proof}
\begin{proof}[Proof of Theorem \ref{T3}.]
It follows directly from Lemmas \ref{L2} and \ref{L5} that \((u_n^*)\) and \((z_n^*)\) are equivalent.
\end{proof}
\begin{Cor} \label{C3}
\(Z\) is isomorphic to a quotient of \(X_{\mathcal{M}}\) and the map
\(Q \colon  X_{\mathcal{M}} \to Z\) given by
\[Q \biggl ( \sum_n a_n e_n \biggr ) = \sum_n \biggl( \sum_{k \in F_n} a_k /|F_n| \biggr ) z_n \]
is a well-defined, bounded, linear surjection.
\end{Cor}
\begin{proof}
Since \((z_n^*)\) dominates \((u_n^*)\), we have that \(Q\) is a well-defined, bounded, linear operator.
It is easy to see now, that \(Q^*(z_n^*) = u_n^*\), for all \(n \in \mathbb{N}\), and thus
\(Q^*\) is an isomorphic embedding of \(Z^*\) into \(X_{\mathcal{M}}^*\), by Theorem \ref{T3}.
It follows now that \(Q\) is a surjection.
\end{proof}

\end{document}